\documentclass{amsart}
\usepackage[utf8]{inputenc}

\usepackage{xcolor}

\usepackage{amsthm}
\usepackage{amsfonts}         
\usepackage{amsmath}
\usepackage{amssymb}
\usepackage{enumerate}
\usepackage[shortlabels]{enumitem}
\usepackage[normalem]{ulem}
\usepackage{comment}
\usepackage[margin=1in]{geometry}
\usepackage{hyperref}
\usepackage{graphicx}
\usepackage{caption}
\usepackage{subcaption}
\usepackage{textcomp}

\newcommand{\R}{\mathbb{R}}
\newcommand{\C}{\mathbb{C}}

\newcommand{\N}{\mathbb{N}}

\newcommand{\vspan}{\operatorname{span}}
\DeclareMathOperator{\sign}{sign}
\DeclareMathOperator{\tr}{tr}

\DeclareMathOperator{\Li}{Li}
\renewcommand{\d}[1]{\ensuremath{\operatorname{d}\!{#1}}}

\renewcommand{\Im}{\operatorname{Im}}
\renewcommand{\Re}{\operatorname{Re}}
\renewcommand{\sign}{\operatorname{sgn}}
\newcommand{\biglamb}{\mathrm{big}}

\theoremstyle{plain}
\newtheorem{lause}[subsection]{Theorem}
\newtheorem{lem}[subsection]{Lemma}
\newtheorem{prop}[subsection]{Proposition}
\newtheorem{kor}[subsection]{Corollary}

\newtheorem{huom}[subsection]{Remark}

\newenvironment{varproof}
 {\proof}
 {\endproof}

\theoremstyle{definition}

\title{Tracial Joint Spectral Measures}
\author{Otte Heinävaara}
\address{Department of Mathematics, Princeton University, Princeton, NJ 08544}
\email{oeh@math.princeton.edu}
\date{\today}

\begin{document}

\begin{abstract}
    Given two Hermitian matrices, $A$ and $B$, we introduce a new type of spectral measure, a \textit{tracial joint spectral measure} $\mu_{A, B}$ on the plane. Existence of this measure implies the following two results: 1) any two-dimensional subspace of the Schatten-$p$ class is isometric to a subspace of $L_{p}$, and 2) if $f : \R \to \R$ has non-negative $k$th derivative and $A$ and $B$ are Hermitian matrices with $A$ positive semidefinite, then $t \mapsto \tr f(t A + B)$ has non-negative $k$th derivative. We also give an explicit expression for the measure $\mu_{A, B}$.
\end{abstract}

\maketitle

\section{Introduction}

For $1\le p\le \infty$, let $S_p$ denote the corresponding Schatten von-Neumann trace class, namely the space of compact linear operators $A:\ell_2\to \ell_2$, equipped with the following norm
\begin{align*}
\|A\|_{S_p}=\left(\tr\left[(A^{*} A)^{\frac{p}{2}}\right]\right)^{\frac{1}{p}}.
\end{align*}
Ball, Carlen, and Lieb conjectured~\cite{ball1994sharp} that if $p\ge 2$, then the following inequality holds for any $A,B\in S_p$,
\begin{equation}\label{eq:hanner intro}
\|A + B\|_{S_{p}}^{p} + \|A - B\|_{S_{p}}^{p} \leq \left(\|A\|_{S_{p}}  + \|B\|_{S_{p}}\right)^{p} + \left|\|A\|_{S_{p}}  - \|B\|_{S_{p}}\right|^{p}.
\end{equation}

\sloppy
This inequality is called (the $S_p$ version of) Hanner's inequality, in reference to Hanner's celebrated work~\cite{hanner1956uniform}, where it was established when $S_p$ is replaced by $\ell_p$ (equivalently, when $A,B$ commute), and this was used to compute the moduli of uniform convexity and uniform smoothness of $\ell_p$; see Section~2.5 in the textbook~\cite{MR1817225} of Lieb and Loss for a treatment of this classical material. 

Hanner's inequality for $S_p$ was proved in~\cite{ball1994sharp}  when $p\ge 4$, and it was proved for $p=3$ in our work~\cite{heinavaara2022planes}.  In addition to the trivial case $p=2$, these were the only values of $p$ for which~\eqref{eq:hanner intro}  was previously known (though, the inequality~\eqref{eq:hanner intro} was established under certain further restrictive assumptions on $A, B$ in~\cite{MR0225140, ball1994sharp, MR4256836}). Here, we will prove~\eqref{eq:hanner intro}  in full generality for any $p\ge 2$, thus settling the Ball--Carlen--Lieb (BCL) conjecture. We will achieve this as a quick corollary of a much more general structural result that we obtain herein. 

\begin{kor}\label{hanner}
    Inequality~\eqref{eq:hanner intro} holds for any $p\ge 2$ and any $A,B\in S_p$. 
\end{kor}

The aforementioned structural result has implications that go far beyond merely proving the BCL conjecture. As another example of a quick corollary of it, we will deduce a conceptually new proof of a conjecture~\cite{MR0416396} of Bessis, Moussa and Villani (BMV), which was proved by Stahl in the celebrated work~\cite{stahl2013proof}; see also~\cite{MR3526447, MR3354963} for expositions and explanations of Stahl's proof, as well as~\cite{MR2981640} for interesting equivalent formulations and implications of Stahl's theorem. Specifically, Stahl's theorem (formerly the BMV conjecture) asserts that the following function is completely monotone for any $n\in \N$ and any two Hermitian $A,B\in M_n(\C)$ such that $A$ is positive semidefinite:
\begin{align}\label{expfun}
    (t\ge 0) \mapsto \tr e^{B - t A}. 
\end{align}
We will soon deduce this as a special case of the  structural result in Theorem~\ref{main3} below, which furthermore implies the following generalization that we do not expect can be deduced using Stahl's approach:   

\begin{kor}\label{traceTone}
    Fix $n\in \N$. Let $A,B\in M_n(\C)$ be Hermitian and let $f:\R\to \R$ be a smooth function with non-negative $k$th derivative. Consider the function $F : \R\to \R$ given by $F(t)= \tr f(t A + B)$. If $k$ is even, then  $F$ is smooth with non-negative $k$th derivative. The same holds for odd $k$ if we additionally assume that $A$ is positive semidefinite.
\end{kor}

Corollary~\ref{traceTone} is well-known for $k = 1$ and $k = 2$ (see for instance \cite[Proposition 1]{petz1994survey}) while for $k = 3$ and $k = 4$ it was proven in \cite{heinavaara2022planes} where the full result was also conjectured. Corollary \ref{traceTone} applied to $f(t) = \exp(-t)$ immediately implies Stahl's theorem.

\subsection{Tracial joint spectral measures} We will now describe our main structural result; further applications are deferred to future works (see Section \ref{future_work}).

\begin{lause}\label{main3}
    Let $n$ be a positive integer and $A, B \in M_n(\C)$ be Hermitian. Then, there exists a positive measure $\mu_{A, B}$ on $\R^{2}$, that we call the {\em tracial joint spectral measure} of $A$ and $B$, such that the following is true:

    Fix any measurable function $f$ on $\R$ such that for any $M > 0$,
    \begin{align*}
        \int_{-M}^{M} \left|\frac{f(t)}{t}\right| \d{t} < \infty.
    \end{align*}
    Define a function $H(f) : \R \to \R$ by
    \begin{align*}
        H(f)(x) &= \int_{0}^{1} \frac{1 - t}{t} f(x t) \d{t}.
    \end{align*}
    Then, for any $x, y \in \R$, we have
    \begin{align}\label{mainIdentity}
        \tr H(f)(x A + y B) = \int_{\R^{2}} f(a x + b y) \d{\mu}_{A, B}(a, b).
    \end{align}
\end{lause}

If in Theorem~\ref{main3} we set $f(t)=|t|^p$ for $p > 0$, then $H(f)(t)=|t|^p/(p (p + 1))$, and the identity we obtain is
\begin{align*}
    \tr \left|x A + y B\right|^{p} = p (p + 1) \int_{\R^{2}} |a x + b y|^{p} \d{\mu}_{A, B}(a, b).
\end{align*}
Note that this is giving us an embedding of the span of $A$ and $B$ in $S_p$ to $L_{p}(\mu_{A, B})$, mapping $A$ to $(a, b) \to a$ and $B$ to $(a, b) \to b$. This embedding is (proportional to) an isometry for every $p > 0$ simultaneously. This confirms a conjecture we made in~\cite{heinavaara2022planes}, where the special case $p=3$ was proved by an entirely different approach. Formally, the above embedding result is stronger than what was conjectured in~\cite{heinavaara2022planes}, which did not ask for an embedding that works simultaneously for every $p > 0$ (and only conjectured the case $p \geq 1$). The realization that such a stronger statement could hold was an essential conceptual starting point for the present work. 

Passing from the case of Hermitian matrices $A,B$ to general two dimensional subspaces of $S_p$ is standard, namely, we have the following more general statement:

\begin{kor}\label{schatten_embedding}
    Let $p > 0$ and let $A, B \in S_{p}$. Then the span of $A$ and $B$ is isometric to a subspace of $L_{p}(\mu)$ for some positive measure $\mu$.
\end{kor}
\begin{proof}
    We have seen that the result holds as a consequence of Theorem~\ref{main3} if $A,B\in M_n(\C)$ are Hermitian. General complex matrices $A,B\in M_n(\C)$ can be reduced to this special case by considering the following Hermitian matrices
    \begin{align*}
        A' = \frac{1}{2^{1/p}}\begin{bmatrix}
            0 & A \\
            A^{*} & 0
        \end{bmatrix},
        \hspace{1cm}
        B' = \frac{1}{2^{1/p}}\begin{bmatrix}
            0 & B \\
            B^{*} & 0
        \end{bmatrix},
    \end{align*}
    and noting that $\|x A' + y B'\|_{S_p} = \|x A + y B\|_{S_p}$ for any $x, y \in \R$.
    
    For general $A,B\in S_p$, by approximating $A$ and $B$ with finite rank operators and applying the finite dimensional result, one sees that for any $k \in \N$ there exists a measure $\mu_{k}$ and a $(1 + 1/k)$-distortion embedding $\vspan(A, B)_{S_{p}} \to L_{p}(\mu_{k})$.
    Take an ultraproduct of these maps with respect to a non-principal ultrafilter $\mathcal{U}$ to get an isometric embedding of $\vspan(A, B)_{S_{p}}$ to the ultraproduct $\left(\prod_{k} L_{p}(\mu_{k})\right)_{\mathcal{U}}$, which is known to be isometric to $L_{p}(\mu)$ for some measure $\mu$; see \cite{MR0305035} for the case $p \geq 1$ and \cite{naor_thesis} for the case $0 < p < 1$.
\end{proof}

This embedding result, combined with Hanner's theorem~\cite{hanner1956uniform}, immediately implies Corollary \ref{hanner}. Much more generally, we have the following ``reduction to commuting'' principle: if an inequality only depends on the $S_{p}$-norms of (real) linear combinations of two complex matrices, then it holds as long as it holds for real diagonal matrices. Thus, any property that only depends on two-dimensional subspaces of $S_{p}$ generalizes directly from $L_{p}$ to $S_{p}$. In particular, $S_{p}$ has the same moduli of uniform convexity and uniform smoothness as $L_{p}$, which is a theorem that was previously established in \cite{ball1994sharp}.

As shown in \cite{heinavaara2022planes}, an analogous embedding is in general impossible for more than two matrices (even for $2\times 2$ Hermitian matrices) whenever $p\in  [1,\infty)\setminus \{2\}$.

If one applies Theorem \ref{main3} to the functions $f(t) = t_+^{k - 1}$ for a positive integer $k$, one quickly arrives at Corollary \ref{traceTone}:

\begin{proof}[Proof of Corollary \ref{traceTone}]
    Smoothness follows from a classical result of Rellich, see~\cite[VII, Theorem 3.9]{MR0203473}. Furthermore, by \cite[Corollary 8]{bullen1971criterion}, it is enough to proof that if $f$ is of the form
    \begin{align}\label{ktone_form}
        t \mapsto p_{k - 1}(t) + \sum_{i = 1}^{M} m_{i} (t - c_{i})_+^{k - 1}
    \end{align}
    where $p_{k - 1}$ is a polynomial of degree at most $k - 1$, $(c_{i})_{i = 1}^{M} \in \R^{M}$, and $(m_{i})_{i = 1}^{M} \in \R_{+}^{M}$, then $\tr f(t A + B)$ is a pointwise limit of functions of the same form. The desired conclusion is clear for the polynomial part, and for the remaining terms we can assume that $M = 1$, $c_{1} = 0$, and $m_{1} = 1$. Also, the case $k = 1$ is classical \cite[Proposition 1]{petz1994survey}, so we may assume that $k \geq 2$.

    Applying Theorem~\ref{main3} for $f(t) = t_+^{k - 1}$ (so that $H(f)(t) = f(t)/(k (k - 1))$) and $(x, y) = (t, 1)$, one gets
    \begin{align*}
        \tr (t A + B)_{+}^{k - 1} = k (k - 1) \int_{\R^{2}} (a t + b)_{+}^{k - 1} \d{\mu}_{A, B}(a, b).
    \end{align*}
    If $k$ is an even integer, the integrand $(a t + b)_{+}^{k - 1}$ is of the form (\ref{ktone_form}) for any $a, b \in \R$, and we are done by the positivity of $\mu_{A, B}$. If $k$ is odd, we further need that the support of $\mu_{A, B}$ is contained in the half plane $\{(a, b) \in \R^{2} \mid a \geq 0\}$, which is the following lemma.
\end{proof}

\begin{lem}\label{mu_prop}
    If $A$ is positive semidefinite, then $\mu_{A, B}$ is supported on the right half-plane $\{(a, b) \in \R^{2} \mid a \geq 0\}$.
\end{lem}
\begin{proof}
    Take any $f$ which is positive for the negative reals and vanishes for non-negative reals; $H(f)$ then has the same property. Applying (\ref{mainIdentity}) for this $f$ and $(x, y) = (1, 0)$ results in
    \begin{align*}
        \tr H(f)(A) = \int_{\R^{2}} f(a) \d{\mu}_{A, B}(a, b).
    \end{align*}
    As the eigenvalues of $A$ are non-negative, the left-hand side of this equation vanishes. Hence, so does the right-hand side, implying the claim.
\end{proof}

Observe that smoothness was not needed in the proof of Corollary \ref{traceTone}, and one could generalize it to $k$-tone/$k$-convex functions (see for instance~\cite{bullen1971criterion} for the relevant definitions).

One can check that the tracial joint spectral measure $\mu_{A, B}$ in Theorem \ref{main3} is necessarily unique away from $0$, see Proposition \ref{mu_unique}. While the exact form of the measure was not important for the above applications, $\mu_{A, B}$ turns out to have a particularly simple expression:

\begin{lause}\label{main4}
    Let $n$, $A$, $B$, and $\mu_{A, B}$ be as in Theorem \ref{main3}. Denote by $\mu_{c} = \mu_{c, A, B}$ and $\mu_{s} = \mu_{s, A, B}$ the continuous and singular parts of $\mu_{A, B}$ w.r.t. the Lebesgue measure $m_{2}$ on $\R^{2}$. We assume some linear combination of $A$ and $B$ is invertible. Then, the continuous part $\mu_{c}$ is given by
    \begin{align*}
        \frac{\d{\mu}_{c}}{\d{m}_{2}}(a, b) = \frac{1}{2 \pi}\sum_{i = 1}^{n} \left|\Im\left( \lambda_{i}\left(\left(I - \frac{a A + b B}{a^{2} + b^{2}}\right) (b A - a B)^{-1} \right)\right) \right|.
    \end{align*}
    Furthermore, if $A$ is invertible and $A^{-1} B$ has $n$ distinct eigenvalues, the singular part $\mu_{s}$ satisfies
    \begin{align*}
        \mu_{s}(\varphi) = \sum_{v \in E(A^{-1} B)} \int_{0}^{1} \frac{1 - t}{t} \varphi\left( \langle A v, v \rangle t, \langle B v, v\rangle\right) \d{t}.
    \end{align*}
    where $E(C)$ denotes a set of normalized eigenvectors of a matrix $C\in M_n(\C)$ and $\varphi$ is a smooth function with compact support that does not contain $0$.
\end{lause}

Figure \ref{fig:test} illustrates the measure $\mu_{A, B}$ for some choices of $A$ and $B$.

\subsection{Future work}\label{future_work}

Our structural result opens interesting research directions. Some of these will be pursued in our forthcoming work \cite{heinavaara2023properties}, which in particular will study the following topics:  

\begin{itemize}
    \item Geometric and regularity properties of $\mu_{c}$, $\mu_{s}$ and their supports.
    \item Further consequences and reformulations of the main identity (\ref{mainIdentity}).
    \item The structure of measures $\mu_{A, B}$ for small matrices $A$ and $B$.
    \item Existence and uniqueness of the measures $\mu_{A, B}$ for compact self-adjoint operators $A$ and $B$.
    \item Optimality results which show various ways in which our main result cannot be improved.
    \item Establishing a relationship to hyperbolic polynomials in the sense of Gårding \cite{gaarding1959inequality}.
\end{itemize}

\section{Notation and conventions}\label{conventions}

We denote the set of (complex valued) Schwartz functions on $\R^{d}$ by $\mathcal{S}(\R^{d})$, and its dual space of tempered distributions by $\mathcal{S}'(\R^{d})$. The Fourier transform of a Schwartz function is defined/normalized with
\begin{align*}
    \mathcal{F}(\varphi)(\xi) = \widehat{\varphi}(\xi) =  \frac{1}{(2 \pi)^{k/2}} \int_{x \in \R^{k}} \varphi(x) e^{-i \langle x, \xi \rangle} \d{x}.
\end{align*}
One has $\mathcal{F}(\partial^{\alpha}\varphi)(\xi) = (i \xi)^{\alpha}\mathcal{F}(\varphi)(\xi)$ for any multi-index $\alpha$. The Fourier transform of a tempered distribution $T$ is defined, as usual, using the pairing $\left(\cdot, \cdot \right) : \mathcal{S}'(\R^{d}) \times \mathcal{S}(\R^{d}) \to \C$, as
\begin{align*}
    \left(\widehat{T}, \varphi\right) = \left(T, \widehat{\varphi}\right).
\end{align*}

Let $A, B \in M_n(\C)$ be Hermitian for some fixed positive integer $n$. The multiset of the eigenvalues of $A$ are denoted by $\lambda(A) = (\lambda_{1}(A), \lambda_{2}(A), \ldots, \lambda_{n}(A))$; the ordering is not important for us. The corresponding set of normalized eigenvectors is denoted by $E(A)$. While these vectors are not unique, we make sure to respect this ambiguity.

We will also make use of notions from the theory of matrix pencils (see for instance \cite{ikramov1993matrix}). Given $A$ and $B$ as before, we say that the pair/pencil $(A, B)$ is non-degenerate if some linear combination of $A$ and $B$ is invertible, i.e. $\det(b A - a B)$ is not zero for every $(a, b) \in \R^{2}$. In this case, the determinant has exactly $n$ roots (with multiplicity) in $\C \mathbb{P}^{1}$, which we will call the roots of the pencil $(A, B)$. Real roots, which we interpret as lines in $\R^{2}$, are called \textit{singular lines}. If a root $(a, b)$ is simple, we may define the corresponding (normalized) eigenvector as the vector in the kernel of $b A - a B$. The set of such vectors (for simple roots) is denoted by $E(A, B)$. Again, these vectors are not unique, but we will respect this ambiguity.

\def\heightRatio{0.30}

\renewcommand\thesubfigure{\arabic{subfigure}}

\begin{figure}
\centering
\begin{subfigure}[t]{.45\textwidth}
  \centering
  \includegraphics[width=1.0\linewidth, height=\heightRatio\paperheight]{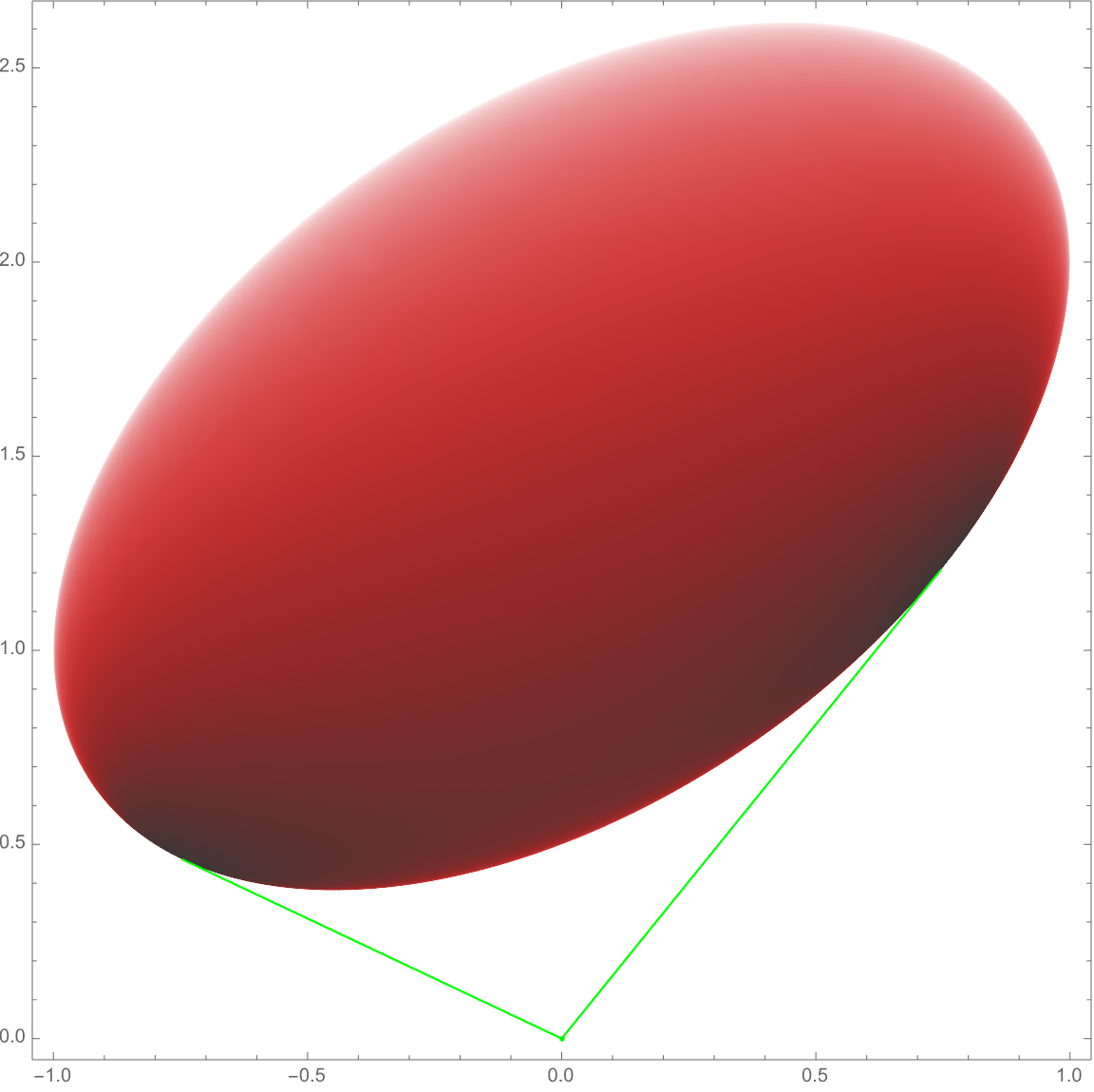}
  \caption{$\left(\begin{bmatrix}
    1 & 0 \\
    0 & -1
  \end{bmatrix}, \begin{bmatrix}
    2 & -1 \\
    -1 & 1
  \end{bmatrix}\right)$.}
  \label{fig:sub1}
\end{subfigure}%
\qquad
\begin{subfigure}[t]{.45\textwidth}
  \centering
  \includegraphics[width=1.0\linewidth, height=\heightRatio\paperheight]{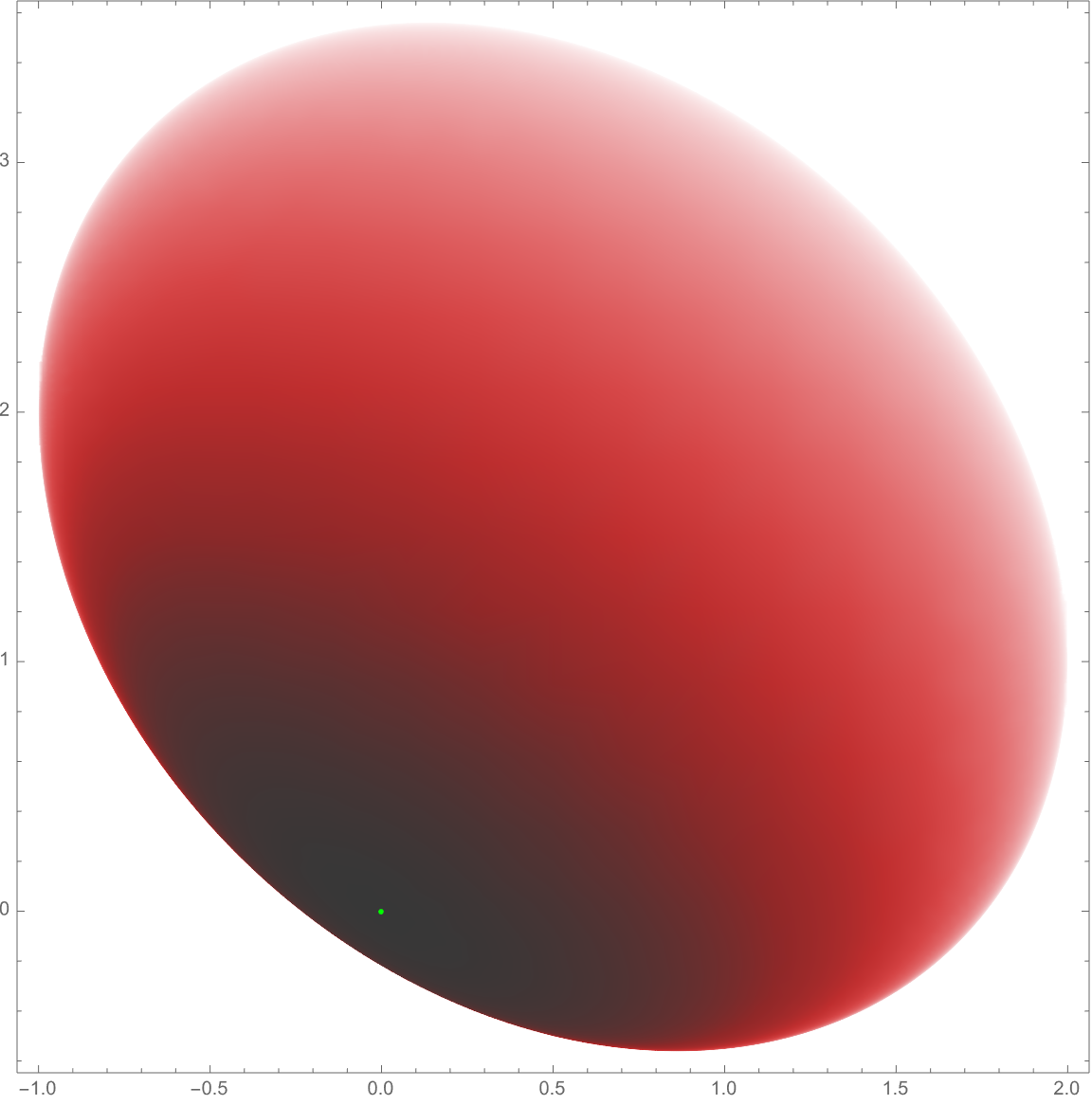}
  \caption{$\left(\begin{bmatrix}
    2 & 0 \\
    0 & -1
  \end{bmatrix}, \begin{bmatrix}
    1 & -2 \\
    -2 & 2
  \end{bmatrix}\right)$.}
  \label{fig:sub2}
\end{subfigure}

\begin{subfigure}[t]{.45\textwidth}
  \centering
  \includegraphics[width=1.0\linewidth, height=\heightRatio\paperheight]{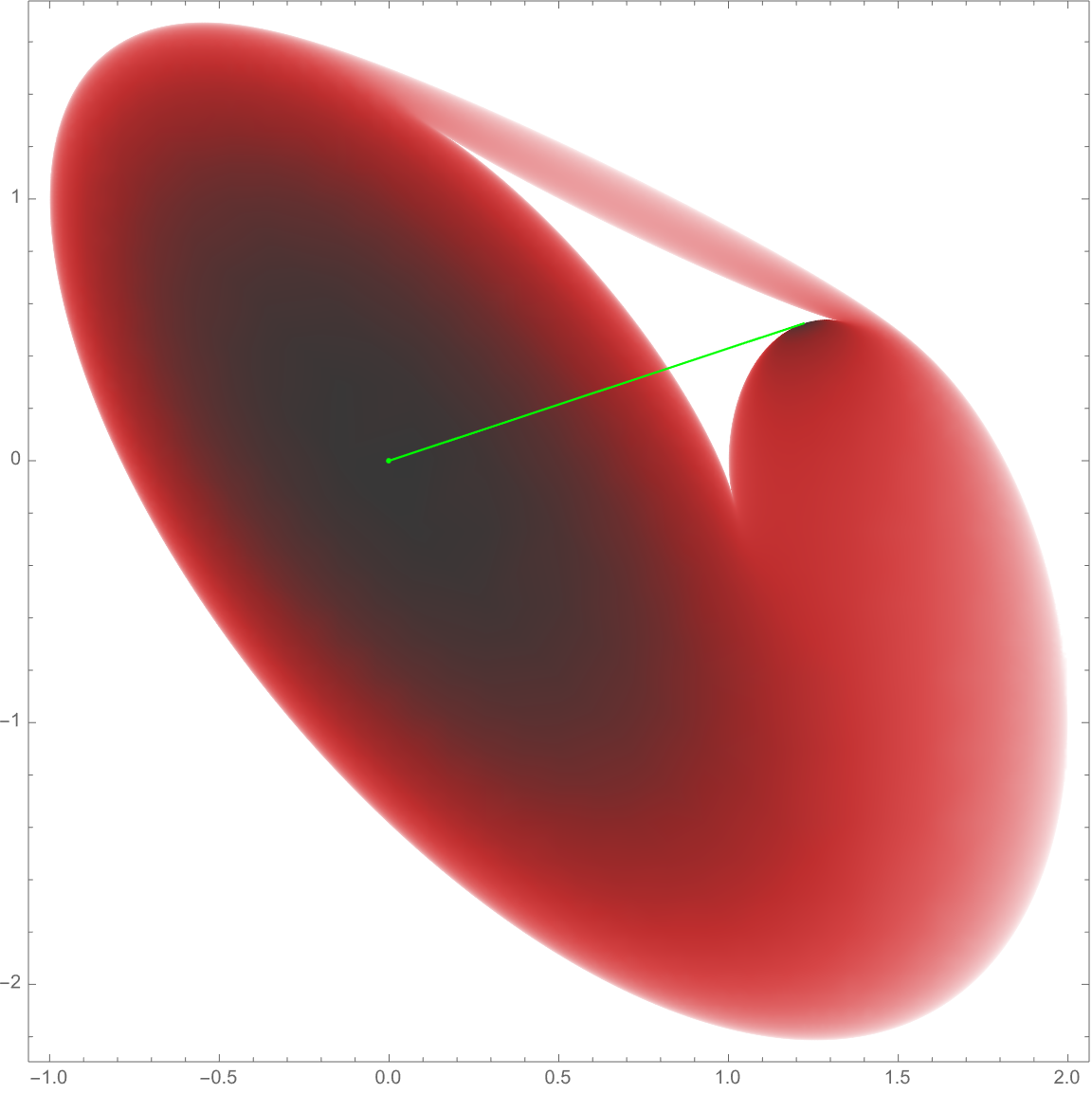}
  \caption{$\left(\begin{bmatrix}
    1 & 0 & 0 \\
    0 & 2 & 0 \\
    0 & 0 & -1
  \end{bmatrix}, \begin{bmatrix}
    0 & 1 & 1 \\
    1 & -1 & -1 \\
    1 & -1 & 1
  \end{bmatrix}\right)$.}
  \label{fig:sub3}
\end{subfigure}%
\qquad
\begin{subfigure}[t]{.45\textwidth}
  \centering
  \includegraphics[width=1.0\linewidth, height=\heightRatio\paperheight]{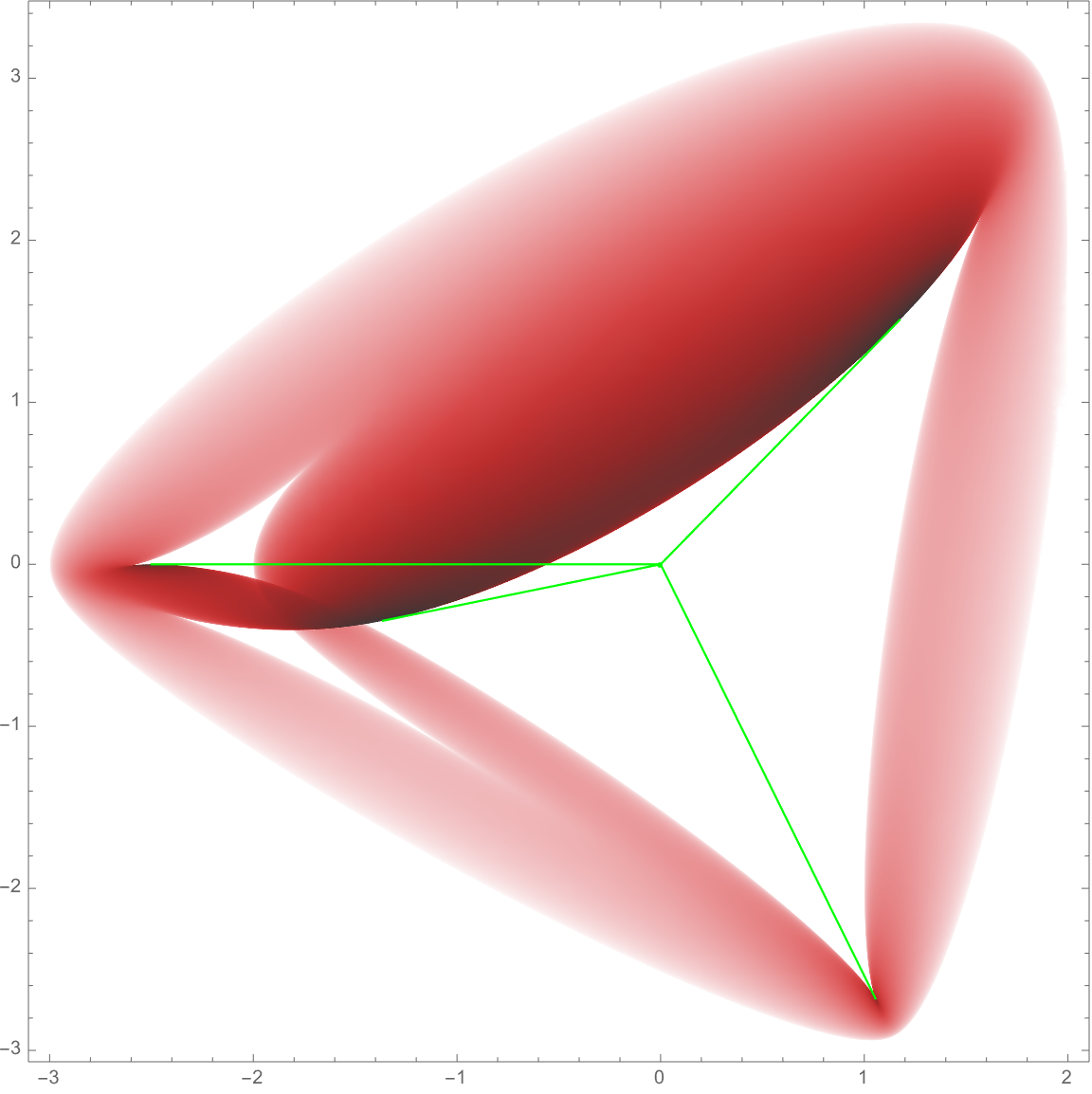}
  \caption{$\left(\begin{bmatrix}
    -3 & 0 & 0 & 0 \\
    0 & -2 & 0 & 0 \\
    0 & 0 & 1 & 0 \\
    0 & 0 & 0 & 2
  \end{bmatrix}, \begin{bmatrix}
    0 & 0 & 0 & -1 \\
    0 & 0 & 0 & -1 \\
    0 & 0 & -2 & -2 \\
    -1 & -1 & -2 & 2
  \end{bmatrix}\right)$.}
  \label{fig:sub4}
\end{subfigure}
\caption{Four illustrations of the measures $\mu_{A, B}$ for the pairs of matrices $(A, B)$ listed below the pictures. The horizontal and vertical axes correspond to $a$ and $b$ respectively. The density of $\mu_{c, A, B}$ is represented with the color running from white (zero) to black (infinity) through red. The green line segments depict the support of the singular part, $\mu_{s, A, B}$.}
\label{fig:test}
\end{figure}

\section{Proofs of the main results}\label{main_proofs}

In this section we will prove slight refinements of Theorems \ref{main3} and \ref{main4}, Theorems \ref{main2} and \ref{main1}.

We start by defining a helpful function $g : \R \to \C$ by
\begin{align}
    g(x) &= \int_{0}^{1} \frac{e^{i x t} - 1}{t} (1 - t) \d{t} = H(y \mapsto e^{i y} - 1)(x). \label{gdef}
\end{align}
Observe that $g$ is smooth, $g(0) = 0$, and $g(x) = O(\log|x|)$ at infinity.

For any two Hermitian $A, B \in M_n(\C)$, define a continuous function $G$ with
\begin{align}
    G := G_{A, B} : \R^{2} &\to \C \label{Gdef} \\
    (x, y) &\to \tr g(x A + y B) \nonumber
\end{align}

We will prove that $\mu_{A, B}$ is essentially the Fourier transform of $G_{A, B}$.

\begin{lause}\label{main1}
    Let $A, B \in M_n(\C)$ be Hermitian. Then there exists a positive measure $\mu_{A, B}$ which agrees with $\widehat{G}_{A, B}$ (see \ref{Gdef}) away from $0$, in the sense that if $\varphi$ is any Schwartz function with compact support not containing $0$, then
    \begin{align*}
        \left(\mu_{A, B}, \varphi\right) = \frac{1}{2 \pi}\left(\widehat{G}_{A, B}, \varphi\right).
    \end{align*}
    Denote by $\mu_{A, B} = \mu_{c} + \mu_{s}$ the Lebesgue decomposition of $\mu_{A, B}$ w.r.t. Lebesgue measure $m_{2}$ ($\mu_{c} \ll m_{2}$, $\mu_{s} \perp m_{2}$). If the pencil $(A, B)$ is non-degenerate, then the continuous part $\mu_{c}$ is given by
    \begin{align}\label{continuous_part}
        \frac{\d{\mu}_{c}}{\d{m}_{2}}(a, b) = \frac{1}{2 \pi}\sum_{i = 1}^{n} \left|\Im\left( \lambda_{i}\left(\left(I - \frac{a A + b B}{a^{2} + b^{2}}\right) (b A - a B)^{-1} \right)\right) \right|.
    \end{align}
    If the pencil $(A, B)$ is non-degenerate and the real roots of $(A, B)$ are distinct, then the singular part $\mu_{s}$ satisfies
    \begin{align}\label{singular_part}
        \mu_{s}(\varphi) = \sum_{v \in E(A, B)} \int_{0}^{1} \frac{1 - t}{t} \varphi\left( \langle A v, v \rangle t, \langle B v, v\rangle t\right) \d{t},
    \end{align}
    where $\varphi$ is Schwartz function with compact support not containing $0$.
\end{lause}

We note that for any $v \in E(A, B)$ corresponding to a root $(a,b)$, $b \langle A v, v\rangle - a \langle B v, v\rangle = 0$, and $\langle A v, v\rangle$ and $\langle B v, v \rangle$ are real. If $v$ further corresponds to a non-real root, then $\langle A v, v\rangle = 0 = \langle B v, v\rangle$. Thus, the sum in (\ref{singular_part}) can equivalently be taken over all eigenvectors of $(A, B)$ that correspond to real roots.

The points $(\langle A v, v\rangle, \langle B v, v\rangle)$ where $v \in E(A, B)$ are called the \textit{singular points}, and they lie on the singular lines, as defined in section \ref{conventions}.

\begin{huom}\label{C_remark}
    There is nothing particularly special about the expression
    \begin{align*}
        C := \left(I - \frac{a A + b B}{a^{2} + b^{2}}\right) (b A - a B)^{-1}.
    \end{align*}
    Since eigenvalues of $C$ and $C + t I$ have equal imaginary parts, we may replace $C$ by anything of the form
    \begin{align*}
        \left(I - f(a, b) A - g(a, b) B\right) (b A - a B)^{-1}
    \end{align*}
    where $f(a, b) a + g(a, b) b = 1$. The expression we chose has the desirable property of making sense for every $(a, b) \neq (0, 0)$, and works well with the change of variables in the proof.
\end{huom}

Before proving Theorem \ref{main1}, we will give a mock proof illustrating our strategy. The major unsound steps are indicated by numbered asterisks. We will comment on how to fix them afterwards.

\begin{varproof}[Mock proof of Theorem \ref{main1}]
    Our goal is to calculate the Fourier transform of $G$. By definition \ref{mistake1} we have
    \begin{align*}
        \widehat{G}(a, b) = \frac{1}{2 \pi}\int_{\R^{2}} e^{-i (a x + b y)} \tr g(x A + y B) \d{x} \d{y} = \frac{1}{2 \pi}\sum_{i = 1}^{n} \int_{\R^{2}} e^{-i (a x + b y)} g(\lambda_{i}(x A + y B)) \d{x} \d{y}.
    \end{align*}
    Rewriting the integral in polar coordinates, we get
    \begin{align*}
        =& \frac{1}{2 \pi}\sum_{i = 1}^{n} \int_{0}^{\pi} \int_{\R} e^{-i r (a \cos(\theta) + b \sin(\theta))} |r| g(\lambda_{i}(\cos(\theta) A + \sin(\theta) B) r) \d{r} \d{\theta} \\
        =& \frac{1}{2 \pi}\sum_{i = 1}^{n} \int_{0}^{\pi} \frac{1}{\lambda_{i}(\cos(\theta) A + \sin(\theta) B)^{2}} \left(\int_{\R} e^{-i r \frac{a \cos(\theta) + b \sin(\theta)}{\lambda_{i}(\cos(\theta) A + \sin(\theta) B)}} |r| g(r) \d{r}\right)\d{\theta} \\
        =& \frac{1}{\sqrt{2 \pi}}\sum_{i = 1}^{n} \int_{0}^{\pi} \frac{\mathcal{F}(x \mapsto |x| g(x))\left(\frac{a \cos(\theta) + b \sin(\theta)}{\lambda_{i}(\cos(\theta) A + \sin(\theta) B)}\right)}{\lambda_{i}(\cos(\theta) A + \sin(\theta) B)^{2}} \d{\theta}.
    \end{align*}
    It turns out that the Fourier transform of $|x|g(x)$ \ref{mistake2} equals
    \begin{align}\label{wrong_fourier}
        \xi \mapsto \sqrt{\frac{2}{\pi}} \frac{1}{\xi^{2}} \log\left|1 - \frac{1}{\xi}\right|,
    \end{align}
    so plugging this in, and multiplying the eigenvalues to obtain the determinant, allows us to simplify to
    \begin{align*}
        =& \frac{1}{\pi} \sum_{i = 1}^{n} \int_{0}^{\pi} \frac{\log \left|1 - \frac{\lambda_{i}(\cos(\theta) A + \sin(\theta) B)}{a \cos(\theta) + b \sin(\theta)}\right|}{(a \cos(\theta) + b \sin(\theta))^{2}} \d{\theta} = \frac{1}{\pi} \int_{0}^{\pi} \frac{\log \left|\det\left(I - \frac{\cos(\theta) A + \sin(\theta) B}{a \cos(\theta) + b \sin(\theta)}\right)\right|}{(a \cos(\theta) + b \sin(\theta))^{2}} \d{\theta}.
    \end{align*}
    Making a change of variable
    \begin{align*}
        t = \frac{1}{a^{2} + b^{2}}\frac{a \cos(\theta) + b \sin(\theta)}{b \cos(\theta) - a \sin(\theta)}
    \end{align*}
    transforms \ref{mistake3} the integral to
    \begin{align*}
        \frac{1}{\pi} \int_{-\infty}^{\infty} \log \left|\det\left(I - \frac{a A + b B}{a^{2} + b^{2}} + t (b A - a B)\right)\right| \d{t}.
    \end{align*}
    Finally, we factorize the determinant and use the fact \ref{mistake4} that
    \begin{align}\label{fake_I1}
        \int_{-\infty}^{\infty} \log|a + b t| \d{t} = \pi \left|\Im\left(\frac{a}{b}\right)\right|
    \end{align}
    to prove the identity (\ref{continuous_part}).
\end{varproof}
\begin{enumerate}[label=(\arabic*$\ast$)]
    \item \label{mistake1} The function $G_{A, B}$ is not integrable, so we will instead calculate its Fourier transform as a distribution, testing against a Schwartz function $\varphi$.
    \item \label{mistake2} The function $|x| g(x)$ does not have a Fourier transform in the usual sense, and in any case (\ref{wrong_fourier}) is not correct. In Lemma \ref{fourier_lemma}, we calculate the correct Fourier transform as a tempered distribution, which is similar to (\ref{wrong_fourier}) but contains some corrections terms.
    \item \label{mistake3} The integrals at hand are not integrable. Instead, we apply a cutoff, split the integral to three parts, and apply a change of variables to each part.
    \item \label{mistake4} The identity (\ref{fake_I1}) is not true, but instead we employ a similar looking identity from Lemma \ref{I1-estimate}.
\end{enumerate}

Additionally, this mock proof doesn't see the singular part. It is hidden (together with the terms making the expressions converge) in the correction terms of the Fourier transform of $|x| g(x)$. These  terms bring complications, as we will need to understand the behaviour of the eigenvalues of $C$ (as in Remark \ref{C_remark}) near the singular lines. These eigenvalue estimates are done in Lemma \ref{eigenvalue_dynamics}.

\begin{proof}[Proof of Theorem \ref{main1}]
Our goal is to calculate the Fourier transform of $G = G_{A, B}$. Fix a Schwartz function $\varphi$. We can rewrite our integral in polar coordinates,
\begin{align*}
    (\widehat{G}, \varphi) &= (G, \widehat{\varphi}) = \int_{\R^{2}} \tr g(x A + y B) \widehat{\varphi}(x, y) \d{x} \d{y} \\
    &= \int_{0}^{\pi} \sum_{i = 1}^{n} \left(\int_{\R} |r| g(\lambda_{i}(\cos(\theta) A + \sin(\theta) B) r) \widehat{\varphi}(r \cos(\theta), r \sin(\theta)) \d{r} \right) \d{\theta}.
\end{align*}
Let $\lambda = \lambda_{i}(\cos(\theta) A + \sin(\theta) B)$. The inner integral vanishes when $\lambda = 0$, and when $\lambda \neq 0$ it equals
\begin{align}
    \int_{\R} |r| g(\lambda r) \widehat{\varphi}(r \cos(\theta), r \sin(\theta)) \d{x} &= \frac{1}{\lambda^{2}}\int_{\R} |r| g(r) \widehat{\varphi}\left(\frac{r \cos(\theta)}{\lambda}, \frac{r \sin(\theta)}{\lambda}\right) \d{r}. \label{phi_integral}
\end{align}
Let $\overline{v}_{\theta} = (\cos(\theta), \sin(\theta))$, $\overline{u}_{\theta} = (-\sin(\theta), \cos(\theta))$, and $\varphi(\overline{w}) = \varphi(\overline{w}_{1}, \overline{w}_{2})$, for any $\overline{w} \in \R^{2}$. By direct calculation, one sees that the term in the integrand
\begin{align*}
    r \mapsto \widehat{\varphi}\left(\frac{r\cos(\theta)}{\lambda}, \frac{r \sin(\theta)}{\lambda}\right)
\end{align*}
is the Fourier transform of the function mapping $x$ to
\begin{align*}
    \frac{\lambda^{2}}{\sqrt{2 \pi}} \int_{\R} \varphi(\lambda x \cos(\theta) - \lambda y \sin(\theta), \lambda x \sin(\theta) + \lambda y \cos(\theta)) \d{y} = \frac{\lambda^{2}}{\sqrt{2 \pi}}  \int_{\R} \varphi(\lambda x \overline{v}_{\theta} + \lambda y \overline{u}_{\theta}) \d{y} =: \frac{\lambda^{2}}{\sqrt{2 \pi}} \phi_{\theta, \lambda}.
\end{align*}
Observe also that $\phi_{\theta, \lambda}(x) = \phi_{\theta, 1}(\lambda x)/|\lambda| =: \phi_{\theta}(\lambda x)/|\lambda|$. Consequently, the integral in (\ref{phi_integral}) simplifies to
\begin{align*}
    \frac{1}{\lambda^{2}}\int_{\R} |r| g(r) \widehat{\varphi}\left(\frac{r}{\lambda} \overline{v}_{\theta}\right) \d{r} &= \frac{1}{\sqrt{2 \pi}} \int_{\R} |r| g(r) \widehat{\phi_{\theta, \lambda}}(r) \d{r} = \frac{1}{\sqrt{2 \pi}} \left(\widehat{|r| g(r)}, \phi_{\theta, \lambda}\right),
\end{align*}
where by $\widehat{|r| g(r)}$ we mean the Fourier transform of the tempered distribution $r \mapsto |r| g(r)$.

\begin{lem}\label{fourier_lemma}
    For any Schwartz function $\phi$ on $\R$, one has
    \begin{align}\label{distributional_identity}
        (\widehat{|x| g(x)}, \phi) = \sqrt{\frac{2}{\pi}} \int_{\R} (\phi(x) - \phi(0) - \phi'(0) x) \frac{1}{x^{2}}\log \left|1 - \frac{1}{x}\right| \d{x}.
    \end{align}
\end{lem}
\begin{proof}
    If $\phi(x) = x^{2} \psi(x)$ for some Schwartz function $\psi$, we have
    \begin{align*}
        (\widehat{|x| g(x)}, \phi) &= (\widehat{-\partial^{2}_{x}|x| g(x)}, \psi) = \left(\widehat{\frac{1 - e^{i x}}{|x|}}, \psi \right) = \sqrt{\frac{2}{\pi}}\left(\log \left|1 - \frac{1}{x}\right|, \psi\right) = \sqrt{\frac{2}{\pi}}\left(\frac{1}{x^{2}}\log \left|1 - \frac{1}{x}\right|, \phi\right),
    \end{align*}
    as desired. For the third equality, see \cite[p. 361]{MR3469458}.

    The result is thus true up to a multiple of $\delta_{0}$ and $\delta_{0}'$. To take care of them, consider the function
    \begin{align*}
        \phi_{\varepsilon, \alpha, \beta}(x) = (\alpha + \beta x) e^{-\frac{1}{2}\varepsilon^{2} x^{2}}.
    \end{align*}
    It is enough to check that for any fixed $\alpha, \beta \in \R$, when $\varepsilon \to 0^{+}$, both the left- and right-hand side of (\ref{distributional_identity}) tend to 0. For the right-hand side, this follows from the dominated convergence theorem. For the left-hand side, note that
    \begin{align*}
        (\widehat{|x| g(x)}, \phi_{\varepsilon, \alpha, \beta}) = (|x| g(x), \widehat{\phi_{\varepsilon, \alpha, \beta}}) &= \int_{\R} |x| g(x) \left(\frac{\alpha}{\varepsilon} + i \frac{\beta}{\varepsilon^{3}} x \right) e^{-\frac{1}{2 \varepsilon^{2}} x^{2}} \d{x} \\
        &= \int_{\R} |x| g(\varepsilon x) \left(\varepsilon \alpha + i \beta x \right) e^{-\frac{1}{2} x^{2}} \d{x},
    \end{align*}
    which tends to $0$ as $\varepsilon \to 0^+$, since $g$ is continuous and $g(0) = 0$.
\end{proof}
Using Lemma~\ref{fourier_lemma}, we can write
\begin{align*}
    \frac{1}{\sqrt{2 \pi}} \left(\widehat{|r| g(r)}, \phi_{\theta, \lambda}\right) &= \frac{1}{\pi} \int_{\R} (\phi_{\theta, \lambda}(r) - \phi_{\theta, \lambda}(0) - \phi_{\theta, \lambda}'(0) r) \frac{1}{r^{2}}\log \left|1 - \frac{1}{r}\right| \d{r} \\
    &= \frac{1}{\pi} \int_{\R} (\phi_{\theta}(x) - \phi_{\theta}(0) - \phi_{\theta}'(0) x) \frac{1}{x^{2}}\log \left|1 - \frac{\lambda}{x}\right| \d{x}.
\end{align*}

Recall that $\lambda$ was an eigenvalue of $\cos(\theta) A + \sin(\theta) B$. We can now sum the above expression over all the eigenvalues of $\cos(\theta) A + \sin(\theta) B$, and integrate over $\theta$. The eigenvalues multiply to form the determinant, and we obtain
\begin{align*}
    (\widehat{G}, \varphi) = \frac{1}{\pi} \int_{0}^{\pi} \int_{\R} (\phi_{\theta}(x) - \phi_{\theta}(0) - \phi_{\theta}'(0) x) \frac{\log \left|\det\left(I - \frac{\cos(\theta) A + \sin(\theta) B}{x} \right)\right|}{x^{2}} \d{x} \d{\theta}.
\end{align*}
We would like to split this integral to the three parts corresponding to $\phi_{\theta}(x)$, $\phi_{\theta}(0)$, and $\phi'_{\theta}(0)$. While the resulting parts don't converge, we can remedy this with a cutoff:
\begin{align*}
    (\widehat{G}, \varphi) &= \lim_{\varepsilon \to 0^{+}} \left(\frac{1}{\pi} \int_{0}^{\pi} \int_{|x| > \varepsilon} \phi_{\theta}(x) \frac{\log \left|\det\left(I - \frac{\cos(\theta) A + \sin(\theta) B}{x} \right)\right|}{x^{2}} \d{x} \d{\theta} \right. \\
    &- \frac{1}{\pi} \int_{0}^{\pi} \int_{|x| > \varepsilon} \phi_{\theta}(0) \frac{\log \left|\det\left(I - \frac{\cos(\theta) A + \sin(\theta) B}{x} \right)\right|}{x^{2}} \d{x} \d{\theta} \\
    &- \left. \frac{1}{\pi} \int_{0}^{\pi} \int_{|x| > \varepsilon} \phi'_{\theta}(0) \frac{\log \left|\det\left(I - \frac{\cos(\theta) A + \sin(\theta) B}{x} \right)\right|}{x} \d{x} \d{\theta} \right)
\end{align*}

We will now analyze the three integrals inside the limit for $\varepsilon > 0$; these integrals
are absolutely integrable.

\begin{enumerate}[(i)]
    \item \textbf{The $\phi_{\theta}(x)$-term:} By definition of $\phi_\theta$,
    \begin{align*}
         & \frac{1}{\pi} \int_{0}^{\pi} \int_{|x| > \varepsilon} \phi_{\theta}(x) \frac{\log \left|\det\left(I - \frac{\cos(\theta) A + \sin(\theta) B}{x} \right)\right|}{x^{2}} \d{x} \d{\theta} \\
        =& \frac{1}{\pi} \int_{0}^{\pi} \int_{|x| > \varepsilon} \int_{\R} \varphi(\lambda x \overline{v}_{\theta} + \lambda y \overline{u}_{\theta}) \frac{\log \left|\det\left(I - \frac{\cos(\theta) A + \sin(\theta) B}{x} \right)\right|}{x^{2}} \d{y} \d{x} \d{\theta} 
    \end{align*}
    We make the change of variables
    \begin{align*}
        (a, b, t) = \left(x \cos(\theta) - y \sin(\theta), x \sin(\theta) + y \cos(\theta), \frac{y}{x (x^2 + y^2)}\right),
    \end{align*}
    and get
    \begin{align*}
        \int_{t^{2} (a^{2} + b^{2})^{2} < (a^{2} + b^{2})/\varepsilon^{2} - 1}  \log \left|\det\left(I - \frac{a A + b B}{a^{2} + b^{2}} - t (b A - a B) \right)\right| \varphi(a, b) \d{a} \d{b} \d{t}.
    \end{align*}
    
    \item \textbf{The $\phi_{\theta}(0)$-term:} With a change of variables similar to the previous case,
    \begin{align*}
        (a, b, t) = \left(-\sin(\theta) y, \cos(\theta) y, \frac{1}{x y}\right),
    \end{align*}
    we simplify to
    \begin{align*}
        \int_{t^{2} (a^{2} + b^{2})^{2} < (a^{2} + b^{2})/\varepsilon^{2}}  \log \left|\det\left(I - t (b A - a B) \right)\right| \varphi(a, b) \d{a} \d{b} \d{t}.
    \end{align*}

    \item \textbf{The $\phi'_{\theta}(0)$-term:} With the same change of variables as in the $\phi_\theta(0)$ case, one can simplify to
    \begin{align*}
        \int_{t^{2} (a^{2} + b^{2})^{2} < (a^{2} + b^{2})/\varepsilon^{2}}  \log \left|\det\left(I - t (b A - a B) \right)\right| \frac{\frac{d}{d h}\varphi(a + h b, b - h a) \Big\vert_{h = 0}}{(a^{2} + b^{2}) t} \d{a} \d{b} \d{t}.
    \end{align*}
\end{enumerate}

At this point, we have proven that
\begin{align*}
     \left(\widehat{G}, \varphi\right) =& \lim_{\varepsilon \to 0^{+}} \left(\int_{t^{2} (a^{2} + b^{2})^{2} < (a^{2} + b^{2})/\varepsilon^{2} - 1}  \log \left|\det\left(I - \frac{a A + b B}{a^{2} + b^{2}} - t (b A - a B) \right)\right| \varphi(a, b) \d{a} \d{b} \d{t} \right. \\
     &\left.- \int_{t^{2} (a^{2} + b^{2})^{2} < (a^{2} + b^{2})/\varepsilon^{2}}  \log \left|\det\left(I - t (b A - a B) \right)\right| \varphi(a, b) \d{a} \d{b} \d{t} \right. \\
     & \left.-\int_{t^{2} (a^{2} + b^{2})^{2} < (a^{2} + b^{2})/\varepsilon^{2}}  \log \left|\det\left(I - t (b A - a B) \right)\right| \frac{\frac{d}{d h}\varphi(a + h b, b - h a) \Big\vert_{h = 0}}{(a^{2} + b^{2}) t} \d{a} \d{b} \d{t}\right).
\end{align*}

Write
\begin{align}\label{ci_def}
    C_{1}(a, b) := \left(I - \frac{a A + b B}{a^{2} + b^{2}}\right)(b A - a B)^{-1}, \text{ and } C_{2}(a, b) := (b A - a B)^{-1}.
\end{align}

To prove identity (\ref{continuous_part}), we recall our additional assumption that $(A, B)$ is non-degenerate. We now also need to assume that $\varphi$ has compact support not containing $0$. We can then rewrite our expression as
\begin{align*}
    & \frac{1}{\pi}\lim_{\varepsilon \to 0^{+}} \left(\int_{(a^{2} + b^{2})/\varepsilon^{2} - 1 < t^{2} (a^{2} + b^{2})^{2} < (a^{2} + b^{2})/\varepsilon^{2}}  \log \left|\det\left(C_{1}(a, b) - t\right)\right| \varphi(a, b) \d{a} \d{b} \d{t} \right. \\
    &\left.+ \int_{t^{2} (a^{2} + b^{2}) < 1/\varepsilon^{2}} \log \left|\frac{\det\left(C_{1}(a, b) - t I\right)}{\det(C_{2}(a, b) - t I)}\right| \varphi(a, b) \d{a} \d{b} \d{t} \right. \\
    &\left. - \int_{(a^{2} + b^{2})/\varepsilon^{2} - 1 < t^{2} (a^{2} + b^{2})^{2} < (a^{2} + b^{2})/\varepsilon^{2}} \log \left|\det\left(C_{2}(a, b)\right)\right| \varphi(a, b) \d{a} \d{b} \d{t}\right. \\
    &\left. - \int_{t^{2} (a^{2} + b^{2}) < 1/\varepsilon^{2}} \log \left|\det\left(I - t C_{2}(a, b)^{-1} \right)\right| \frac{\frac{d}{d h}\varphi(a + h b, b - h a) \Big\vert_{h = 0}}{(a^{2} + b^{2}) t} \d{a} \d{b} \d{t}\right).
\end{align*}
We claim that the first term here tends to $0$ as $\varepsilon \to 0^+$. To that end, we integrate first over $t$ and then $a$ and $b$. Observe that for small $\varepsilon$ and fixed $a$ and $b$, the term
$\log \left|\det\left(C_{1}(a, b) - t\right)\right|$ is a sum of functions $\log |t - c|$, where $c$ is an eigenvalue of $C_{1}(a, b)$.  The integral in $t$ is over an interval of length $O(\varepsilon)$ with distance $O(1/\varepsilon)$ from $0$. One checks that such an integral is $ O(\varepsilon \log(1 + |c|) + \varepsilon \log(1/\varepsilon))$, i.e.
\begin{align*}
    &\left|\int_{(a^{2} + b^{2})/\varepsilon^{2} - 1 < t^{2} (a^{2} + b^{2})^{2} < (a^{2} + b^{2})/\varepsilon^{2}}  \log \left|\det\left(C_{1}(a, b) - t\right)\right| \varphi(a, b) \d{a} \d{b} \d{t}\right| \\
    =&~O\left(\int_{\R^{2}}(\varepsilon \sum_{i = 1}^{n}\log(1 + |\lambda_{i}(C_{1}(a, b))|) + \varepsilon \log(1/\varepsilon)) |\varphi(a, b)| \d{a} \d{b}\right).
\end{align*}
Since the eigenvalues of $C_{1}(a, b)$ explode at most polynomially near the singular lines of $C_{1}(a, b)$, these integrals converge to zero. By a similar argument, one sees that the third term converges to zero.

So, we know that if $\varphi$ has compact support not containing $0$, then
\begin{align*}
     \left(\widehat{G}, \varphi\right) =& \frac{1}{\pi}\lim_{\varepsilon \to 0^{+}}\left(\int_{t^{2} (a^{2} + b^{2}) < 1/\varepsilon^{2}} \log \left|\frac{\det\left(C_{1}(a, b) - t I\right)}{\det(C_{2}(a, b) - t I)}\right| \varphi(a, b) \d{a} \d{b} \d{t} \right. \\
      &\left. - \int_{t^{2} (a^{2} + b^{2}) < 1/\varepsilon^{2}} \log \left|\det\left(I - t C_{2}(a, b)^{-1} \right)\right| \frac{\frac{d}{d h}\varphi(a + h b, b - h a) \Big\vert_{h = 0}}{(a^{2} + b^{2}) t} \d{a} \d{b} \d{t}\right).
\end{align*}

We will now integrate out $t$ using the following two computational lemmas.
\begin{lem}\label{I1-estimate}
    For any $\lambda \in \C$, consider the integral
    \begin{align*}
         I_{1}(\lambda, M) := \int_{|t| < M} \log \left|1 - \frac{\lambda}{t}\right|\d{t}.
    \end{align*}
    Then, we have
    \begin{align}
        & I_{1}(\lambda, M) \nonumber \\
        =& M \log \left|1 - \frac{\lambda^{2}}{M^{2}} \right| + \Re(\lambda) \log \left|\frac{\lambda + M}{\lambda - M}\right| + \Im(\lambda) \left(\arctan\left(\frac{M - \Re(\lambda)}{\Im(\lambda)}\right) - \arctan\left(\frac{-M - \Re(\lambda)}{\Im(\lambda)}\right)\right) \label{betterFormula} \\
        =&: \pi |\Im(\lambda)| + E_{1}(\lambda, M), \nonumber
    \end{align}
    where
    \begin{align*}
        E_{1}(\lambda, M) = O\left(M \log \left(1 + \frac{\lambda^{2}}{M^{2}}\right)\right) + |\Im(\lambda)| o(1);
    \end{align*}
    the $o(1)$ term tends to zero with $|\lambda|/M$.
\end{lem}
\begin{proof}
    The first identity is straightforward, if somewhat tedious to verify. The limit of the expression is $\pi|\Im(\lambda)|$, so it remains to prove the error term estimate.
    
    By scaling, we may assume that $M = 1$.
    Since logarithm is locally integrable, it is enough to consider the cases with $|\lambda| \ll 1$ and $|\lambda| \gg 1$. The first two terms of (\ref{betterFormula}) can be estimated via Taylor expansion, while the third term is straightforward.
\end{proof}

\begin{lem}\label{I2-estimate}
    For any $\lambda \in \R$, consider the integral
    \begin{align*}
         I_{2}(\lambda, M) := \int_{|t| < M} \log \left|1 - \frac{t}{\lambda}\right| \frac{\d{t}}{t}.
    \end{align*}
    Then,
    \begin{align*}
        I_{2}(\lambda, M) &=  -\frac{\pi^2}{2}\sign(\lambda) + E_{2}(\lambda, M),
    \end{align*}
    where the $E_{2}(\lambda, M)$ is bounded and tends to zero with $|\lambda|/M$.
\end{lem}
\begin{proof}
    By scaling, we may assume that $\lambda = 1$. We have
    \begin{align*}
        \lim_{M \to \infty} \int_{|t| < M} \log |1 - t| \frac{\d{t}}{t} &= \int_{0}^{\infty} \log \left|\frac{1 - t}{1 + t}\right| \frac{\d{t}}{t} = \int_{0}^{1} \log \left(\frac{1 - t}{1 + t}\right) \frac{\d{t}}{t} + \int_{1}^{\infty} \log \left(\frac{t - 1}{t + 1}\right) \frac{\d{t}}{t} \\
        &= 2 \int_{0}^{1} \log \left(\frac{1 - t}{1 + t}\right) \frac{\d{t}}{t} = -2 \Li_{2}(1) + 2 \Li_{2}(-1) = -\frac{\pi^2}{2}.
    \end{align*}
    Here, $\Li_{2}$ stands for the dilogarithm function, defined as
    \begin{align*}
        \Li_{2}(z) = \int_{0}^{z} \frac{\log(1 - t)}{t} \d{t},
    \end{align*}
    whose properties and special values are well-documented (see \cite{MR2290758}). The error term estimates are straightforward.
\end{proof}

Having integrated out $t$, we are then left with
\begin{align}
    & \left(\widehat{G}, \varphi\right) \nonumber \\
    &= \frac{1}{\pi} \lim_{\varepsilon \to 0^{+}} \left(\int_{\R^{2}} \left[\sum_{i = 1}^{n} I_{1}\left(\lambda_{i}\left(C_{1}(a, b)\right), \frac{1}{\varepsilon \sqrt{a^{2} + b^{2}}}\right) -\sum_{i = 1}^{n} I_{1}\left(\lambda_{i}\left(C_{2}(a, b)\right), \frac{1}{\varepsilon\sqrt{a^{2} + b^{2}}}\right)\right]\varphi(a, b) \d{a} \d{b}\right. \nonumber \\
    -& \left. \int_{\R^{2}} \left[\sum_{i = 1}^{n} I_{2}\left(\lambda_{i}\left(C_{2}(a, b)\right), \frac{1}{\varepsilon \sqrt{a^{2} + b^{2}}}\right)\right]\frac{\frac{d}{d h}\varphi(a + h b, b - h a) \Big\vert_{h = 0}}{(a^{2} + b^{2})} \d{a} \d{b} \right) \nonumber \\
    &=\int_{\R^{2}} \sum_{i = 1}^{n} \left|\Im(\lambda_{i}(C_{1}(a, b)))\right| \varphi(a, b) \d{a} \d{b} \label{mainTerm} \\
    &+ \frac{1}{\pi} \lim_{\varepsilon \to 0^{+}} \left(\int_{\R^{2}} \left[\sum_{i = 1}^{n} E_{1}\left(\lambda_{i}\left(C_{1}(a, b)\right), \frac{1}{\varepsilon \sqrt{a^{2} + b^{2}}}\right) -\sum_{i = 1}^{n} E_{1}\left(\lambda_{i}\left(C_{2}(a, b)\right), \frac{1}{\varepsilon\sqrt{a^{2} + b^{2}}}\right)\right]\varphi(a, b) \d{a} \d{b}\right. \label{singularleft1}  \\
    -& \left. \int_{\R^{2}} \left[\sum_{i = 1}^{n}\left( -\frac{\pi^{2}}{2} \sign(\lambda_{i}(C_{2}(a, b))) + E_{2}\left(\lambda_{i}\left(C_{2}(a, b)\right), \frac{1}{\varepsilon\sqrt{a^{2} + b^{2}}}\right)\right)\right]\frac{\frac{d}{d h}\varphi(a + h b, b - h a) \Big\vert_{h = 0}}{(a^{2} + b^{2})} \d{a} \d{b} \right). \label{singularleft2}
\end{align}

We are finally ready to isolate the continuous part. Indeed, we will assume that $\varphi$ has compact support disjoint from the singular lines. In this case, the eigenvalues of $C_{1}(a, b)$ and $C_{2}(a, b)$ are bounded uniformly on the support of $\varphi$, so the error terms $E_{1}$ and $E_{2}$ tend uniformly to zero. Additionally, the sign term is locally constant, so integrating it by parts against the derivative term along circular arcs yields $0$. We are left with the first term, (\ref{mainTerm}), which is the desired continuous part (after dividing by $2 \pi$).

It remains to work out the singular part. We recall the additional assumption that $(A, B)$ does not have repeated real roots. If $(A, B)$ does not have real roots, there is no singular part. Otherwise, we can assume that $(0, 1) \in \R \mathbb{P}^{1}$ is one of the roots. This means that $B$ is singular, and the $a$-axis is a singular line; it suffices to consider $\varphi$ for which the support hits only this single singular line. In the following lemma, we analyze the behaviour of the eigenvalues of $C_{1}(a, b)$ and $C_{2}(a, b)$ near this singular line.

\begin{lem}\label{eigenvalue_dynamics}
    Assume that the pencil $(A, B)$ has a simple root $(0, 1)$ with unit eigenvector $v$. Then $\langle A v, v\rangle \neq 0$, and for $a \neq 0$, the matrices $C_{1}(a, b)$ and $C_{2}(a, b)$ (as defined in (\ref{ci_def})) have big eigenvalues, with asymptotics as follows,
    \begin{align*}
        \lambda_{\biglamb}\left(C_{1}(a, b)\right) &= \frac{a - \langle A v, v\rangle}{a b \langle A v, v\rangle} (1 + O(b^{1/n})), \\
        \lambda_{\biglamb}\left(C_{2}(a, b)\right) &= \frac{1}{b \langle A v, v\rangle} (1 + O(b)).
    \end{align*}
    The $O$-terms are uniform for $(a,b) \in K \times (-\delta,\delta)$, where $K$ is compact and does not contain $0$ and $\delta$ is sufficiently small. All the other eigenvalues of $C_{1}(a, b)$ and $C_{2}(a, b)$ are $O(b^{1/n - 1})$ and $O(1)$ respectively, with the same uniformity properties.
\end{lem}
\begin{proof}
    By our assumption, $\det(x A + B)$ has a single zero at $0$. By expanding this determinant in an eigenbasis for $B$, we can see that $\langle A v, v\rangle \neq 0$.

    Observe that
    \begin{align*}
        C_{2}(a, b) = \frac{1}{b \langle A v, v\rangle} v v^{*} + O(1),
    \end{align*}
    where the error term is Hermitian with uniformly bounded entries. This follows at once from Cramer's rule, applied in an eigenbasis for $B$. Consequently, for $C_{1}(a, b)$ we have
    \begin{align*}
        C_{1}(a, b) = \frac{1}{b \langle A v, v\rangle}\left(I - \frac{a A + b B}{a^{2} + b^{2}}\right) v v^{*} + O(1),
    \end{align*}
    where the error term again has uniformly bounded entries but not necessarily Hermitian. The main terms have the desired eigenvalues, and the error estimates follow from well-known eigenvalue perturbation bounds for general ($C_1$) and Hermitian ($C_2$) matrices; see \cite{MR0203473}.
\end{proof}

These eigenvalue estimates imply that (\ref{mainTerm}) is indeed integrable.

We will evaluate (\ref{singularleft1}) and (\ref{singularleft2}) using the eigenvalue expansions, starting with the former. Making the change of variable $b = c \varepsilon$, and observing that $E_{1}(s, t) = E_{1}(\varepsilon s, \varepsilon t)/\varepsilon$, we can rewrite (\ref{singularleft1}) as
\begin{align*}
    \frac{1}{\pi}\int_{\R^{2}} \sum_{i = 1}^{n} \left[E_{1}\left( \varepsilon \lambda_{i}\left(C_{1}(a, \varepsilon c)\right), \frac{1}{\sqrt{a^{2} + \varepsilon^{2} c^{2}} }\right) - E_{1}\left(\varepsilon \lambda_{i}\left(C_{2}(a, \varepsilon c)\right), \frac{1}{\sqrt{a^{2} + \varepsilon^{2} c^{2}}}\right)\right] \varphi(a, \varepsilon c) \d{a} \d{c}.
\end{align*}
By Lemma (\ref{eigenvalue_dynamics}), the integrand converges pointwise, and we obtain
\begin{align*}
    \frac{1}{\pi}\int_{\R^{2}} \left[I_{1}\left(\frac{a - \langle A v, v\rangle}{c a \langle A v, v\rangle}, \frac{1}{|a|}\right) - I_{1}\left(\frac{1}{c \langle A v, v\rangle}, \frac{1}{|a|}\right)\right] \varphi(a, 0) \d{a} \d{c}.
\end{align*}
To justify taking the limit inside, it suffices to note that by Lemma \ref{eigenvalue_dynamics}, for $i \in [n]$, we have $\varepsilon |\lambda_{i}\left(C_{1}(a, \varepsilon c)\right)| = O(1/|c|)$, with imaginary part $O(|c|^{1/n - 1} \varepsilon^{1/n})$. Apply the estimates of Lemma (\ref{I1-estimate}) to get an integrable majorant $c \mapsto C (\log(1 + 1/c^{2}) + |c|^{1/n - 1})$ for some $C > 0$.

We may now use Lemma \ref{I1-estimate} to evaluate the $I_1$-terms to get
\begin{align*}
    & \frac{1}{\pi}\int_{\R^{2}} \left[\frac{a - \langle A v, v\rangle}{c |a| \langle A v, v\rangle} \log\left|\frac{\frac{a - \langle A v, v\rangle}{\langle A v, v\rangle} + c}{\frac{a - \langle A v, v\rangle}{\langle A v, v\rangle} -c} \right| + \log \left|1 - \frac{1}{c^{2}}\frac{(a - \langle A v, v\rangle)^{2}}{\langle A v, v\rangle^{2}}\right|\right] \varphi(a, 0) \d{a} \d{c} \\
    -& \frac{1}{\pi}\int_{\R^{2}} \left[\frac{a}{c |a| \langle A v, v\rangle} \log\left|\frac{\frac{a}{\langle A v, v\rangle} + c}{\frac{a}{\langle A v, v\rangle} -c} \right| + \log \left|1 - \frac{1}{c^{2}}\frac{a^{2}}{\langle A v, v\rangle^{2}}\right|\right] \varphi(a, 0) \d{a} \d{c}.
\end{align*}
Finally, use Lemma \ref{I1-estimate} and \ref{I2-estimate} to calculate the integral in $c$, ending up with
\begin{align*}
    & \pi \int_{\R} \left(\left|\frac{1}{a} - \frac{1}{\langle A v, v\rangle} \right| - \frac{1}{|\langle A v, v\rangle|} \right) \varphi(a, 0) \d{a}.
\end{align*}

We will now turn to (\ref{singularleft2}). Since the $E_{2}$-term is bounded and converges to zero, it vanishes in the limit, and we are left with
\begin{align*}
    &\frac{\pi}{2}\int_{\R^{2}} \left[\sum_{i = 1}^{n} \sign(\lambda_{i}(C_{2}(a, b)))\right]\frac{\frac{d}{d h}\varphi(a + h b, b - h a) \Big\vert_{h = 0}}{(a^{2} + b^{2})} \d{a} \d{b} \\
    =& \frac{\pi}{2}\int_{\R^{2}} \sign(\langle A v, v\rangle b) \frac{\frac{d}{d h}\varphi(a + h b, b - h a) \Big\vert_{h = 0}}{(a^{2} + b^{2})} \d{a} \d{b}.
\end{align*}
Here, the equality follows from the fact that only the big eigenvalue can change its sign in the support of $\varphi$, and its sign is determined by Lemma \ref{eigenvalue_dynamics}. This integral can be further simplified by integration along the $\sqrt{a^{2} + b^{2}}$-radius arcs, with say the change of variables $(a, b) = (r \cos(\theta), r \sin(\theta))$. This results in the integral
\begin{align*}
    & \pi \int_{\R} \frac{\sign(\langle A v, v\rangle)}{a} \varphi(a, 0) \d{a}.
\end{align*}

Putting the terms together, we can see that the singular part is given by
\begin{align*}
    & \pi \int_{\R} \left(\left|\frac{1}{a} - \frac{1}{\langle A v, v\rangle} \right| - \frac{1}{|\langle A v, v\rangle|} + \frac{\sign(\langle A v, v\rangle)}{a} \right) \varphi(a, 0) \d{a} \\
    &= 2 \pi \int_{0}^{1} \frac{1 - t}{t} \varphi(\langle A v, v\rangle t, 0) \d{t} = 2 \pi \int_{0}^{1} \frac{1 - t}{t} \varphi(\langle A v, v\rangle t, \langle B v, v\rangle) \d{t},
\end{align*}
as desired.

It remains to get rid of the extra assumptions for the existence of $\mu_{A,B}$. This can be done with approximation: one can find a sequence of pairs $(A_{m}, B_{m})$ converging to $(A, B)$, such that 1) pencils $(A_{m}, B_{m})$ are non-degenerate, and 2) all roots of $(A_{m}, B_{m})$ are pairwise distinct.
These conditions are Zariski open, so are satisfied by small generic perturbations. Then, $(\widehat{G_{A_{m}, B_{m}}}, \varphi) = (G_{A_{m}, B_{m}}, \widehat{\varphi}) \to (G_{A, B}, \widehat{\varphi}) = (\widehat{G_{A, B}}, \varphi)$ for any $\varphi$ as before. So, $\widehat{G_{A,B}}$ is a weak limit of positive measures and hence a positive measure itself.
\end{proof}

\begin{lause}\label{main2}
    Let $\mu_{A, B}$ be as in Theorem \ref{main1}. Fix any measurable function $f$ such that for any $M > 0$,
    \begin{align*}
        \int_{-M}^{M} \left|\frac{f(t)}{t}\right| \d{t} < \infty.
    \end{align*}
    Define a function $H(f) : \R \to \R$ by 
    \begin{align*}
        H(f)(x) &= \int_{0}^{1} \frac{1 - t}{t} f(x t) \d{t}.
    \end{align*}
    Then, for any $x, y \in \R$, one has
    \begin{align}\label{main_new}
        \tr H(f)(A x + By) = \int_{\R^{2}} f(a x + b y) \d{\mu}_{A, B}(a, b).
    \end{align}
\end{lause}
\begin{proof}
    Let us start by considering a Schwartz function $f$ with compact support not containing $0$. By a change of variables (see Proposition \ref{basic_properties}, (\ref{lol_arg})), we may assume that $(x, y) = (1, 0)$. Define $\varphi_{\varepsilon}(a, b) = f(a) e^{-1/2 \varepsilon^{2} b^{2}}$. By the defining property of the measure $\mu_{A,B}$ from Theorem~\ref{main1},
    \begin{align*}
        \int_{\R^{2}} f(a) \d{\mu}_{A, B}(a, b) =& \lim_{\varepsilon \to 0^{+}}\int_{\R^{2}} \varphi_{\varepsilon}(a, b) \d{\mu}_{A, B}(a, b) = \lim_{\varepsilon \to 0^{+}} \int_{\R^{2}} \widehat{\varphi}_{\varepsilon}(a, b) G_{A, B}(a, b)\d{a} \d{b} \\
        =& \lim_{\varepsilon \to 0^{+}} \frac{1}{\varepsilon}\int_{\R^{2}} \widehat{f}(a) e^{-\frac{1}{2 \varepsilon^{2}} b^{2}} \tr g(a A + b B) \d{a} \d{b} = \lim_{\varepsilon \to 0^{+}} \int_{\R^{2}} \widehat{f}(a) e^{-\frac{1}{2} b^{2}} \tr g(a A + c \varepsilon B) \d{a} \d{c}\\
        =& \int_{\R^{2}} \widehat{f}(a) e^{-\frac{1}{2} b^{2}} \tr g(a A) \d{a} \d{c} = \sqrt{2 \pi} \sum_{i = 1}^{n} \int_{\R} \widehat{f}(a) g(a \lambda_{i}(A)) \d{a} \\
        =& \sqrt{2 \pi} \sum_{i = 1}^{n} \int_{\R} \mathcal{F}(x \mapsto f(\lambda_{i}(A) x))(a) g(a) \d{a} = \sqrt{2 \pi} \sum_{i = 1}^{n} \left(\widehat{g}, x \mapsto f(\lambda_{i}(A) x)\right).
    \end{align*}
    It therefore suffices to check that
    \begin{align*}
        \int_{0}^{1} \frac{1 - t}{t} f(t) \d{t} &= \frac{1}{\sqrt{2 \pi }}\left(\hat{g}, f\right).
    \end{align*}
    Writing $f(t) = t^{2} h(t)$, we are left to verify that
    \begin{align*}
        \int_{0}^{1} t (1 - t) h(t) \d{t} &= -\frac{1}{\sqrt{2 \pi }}\left(\widehat{\partial^{2} g}, h\right) = -\frac{1}{\sqrt{2 \pi }}\left(\mathcal{F}\left(x \mapsto\frac{(x + 2 i) e^{i x}}{x^{3}} + \frac{x - 2 i}{x^{3}}\right), h\right).
    \end{align*}
    But this is straightforward to check by calculating the inverse Fourier transform of $t (1 - t) \chi_{[0, 1]}(t)$.

    A general $f$ can be dealt with approximation. Start by assuming that $f$ is bounded and compactly supported with the support not containing $0$. One can then find a sequence $(f_{m})_{m = 1}^{\infty}$ of Schwartz functions with the same bound converging pointwise a.e. to $f$. Dominated convergence theorem then implies that both sides of (\ref{main_new}) converge when $m \to \infty$, so the identity (\ref{main_new}) is also true for such an $f$. A general non-negative $f$ can be now dealt with monotone convergence theorem, and to finish, decompose $f$ to positive and negative parts.
\end{proof}

\begin{prop}\label{basic_properties}
    \begin{enumerate}
        \item (Basis change) Let $V : \R^{2} \to \R^{2}$ be linear and invertible with  $V(a, b) = (v_{1, 1} a + v_{1, 2} b, v_{2, 1} a + v_{2, 2} b)$ for $a, b \in \R$. Define $(A', B') = V(A, B) = (v_{1, 1} A + v_{1, 2} B, v_{2, 1} A + v_{2, 2} B)$. Then, $\mu_{A', B'}$ is given by the pushforward measure $V_{*}(\mu_{A, B})$.
        \item (Invariance) The measure $\mu_{A, B}$ only depends on the homogeneous polynomial, the so called Kippenhahn polynomial,
        \begin{align*}
            p_{A, B}(x, y, z) = \det(z I + x A + y B),
        \end{align*}
        in the sense that if $p_{A, B} = p_{A', B'}$ for a different pair $(A', B')$, then $\mu_{A, B} = \mu_{A', B'}$.
        \item (Block matrices) Assume that $(A, B)$ is block diagonal, i.e. in some basis we have
        \begin{align*}
            A = \begin{bmatrix}
                A_{1} & 0 \\
                0 & A_{2}
            \end{bmatrix}
            \text{ and }
            B = \begin{bmatrix}
                B_{1} & 0 \\
                0 & B_{2}
            \end{bmatrix}
        \end{align*}
        for some $(A_{1}, B_{1}) \in M_{n_1}(\C) \times M_{n_1}(\C)$ and $(A_{2}, B_{2}) \in M_{n_2}(\C) \times M_{n_2}(\C)$ (with $n_{1} + n_{2} = n$, $n_{1}, n_{2} > 0$). Then,
        \begin{align}\label{mu_factoring}
            \mu_{A, B} = \mu_{A_{1}, B_{1}} + \mu_{A_{2}, B_{2}}.
        \end{align}
        \item If $p_{A, B}$ is reducible, then for some $n_{1}, n_{2} > 0$ with $n_{1} + n_{2} = n$, there exists  $(A_{1}, B_{1}) \in M_{n_1}(\C) \times M_{n_1}(\C)$ and $(A_{2}, B_{2})\in M_{n_2}(\C) \times M_{n_2}(\C)$ such that $p_{A, B} = p_{A_{1}, B_{1}} p_{A_{2}, B_{2}}$ and  hence (\ref{mu_factoring}) holds.
    \end{enumerate}
\end{prop}
\begin{proof}
    \begin{enumerate}
        \item \label{lol_arg} It follows by manipulating (\ref{main_new}) that $V_{*}(\mu_{A, B})$ satisfies the defining identity (\ref{main_new}) for $\mu_{A', B'}$. Hence, we will be done by uniqueness of the measure, Proposition \ref{mu_unique}.

        Alternatively, this follows from the explicit formulas (\ref{continuous_part}) and (\ref{singular_part}). Indeed, Remark \ref{C_remark} implies that
        \begin{align*}
            \rho_{A', B'} = \frac{1}{|\det(V)|} \rho_{A, B} \circ V^{-1},
        \end{align*}
        which is exactly the density of the pushforward of the continuous part of $\mu_{A, B}$. The singular points are respected by the pushforward, and hence the singular part in entirety.
        \item This is clear since the left-hand side of (\ref{main_new}) only depends on the eigenvalues (with multiplicities) of linear combinations of $A$ and $B$; and these are the same for $(A, B)$ and $(A', B')$ if $p_{A, B} = p_{A', B'}$.
        \item This follows from $\tr f(x A + yB) = \tr f(x A_{1} + y B_{1}) + \tr f(x A_{2} + y B_{2})$, and uniqueness of the measure.
        \item \label{helton_vinnikov} The polynomial $p_{A, B}$ and hence all its factors are hyperbolic in the sense of Gårding \cite{gaarding1959inequality}. The existence of $(A_{1}, B_{1})$ and $(A_{2}, B_{2})$ then follows from the Helton--Vinnikov theorem \cite{helton2007linear}.
    \end{enumerate}
\end{proof}

Property (\ref{helton_vinnikov}) of Proposition~\ref{basic_properties} implies that degenerate pencils can be reduced to the non-degenerate case by factoring the polynomial $p_{A, B}$.

\begin{prop}\label{mu_unique}
    There is at most one measure $\mu$ satisfying the condition of Theorem \ref{main2}.
\end{prop}
\begin{proof}
    Fix $M > \|A\| + \|B\|$, and consider $f$ that vanishes on $[-M, M]$ and is positive for $|x| > M$. Then $H(f)$ also vanishes on $[-M, M]$ and is positive for $|x| > M$, so the left-hand side of (\ref{main_new}) vanishes whenever $x, y \in [-1, 1]$. On the right-hand side, we see that $\mu$ is supported on $\{(a, b) \mid |a| + |b| \leq M\}$. While $\mu$ is not a finite measure, the measure $\tilde{\mu}$ defined by
    \begin{align*}
        \tilde{\mu}(f) = \int_{\R^2} (a^{2} + b^{2}) f(a, b) \d{\mu}(a, b)
    \end{align*}
    is. Applying (\ref{main_new}) for different polynomials and $x, y \in \R$, one can fix all moments of $\tilde{\mu}$. As $\tilde{\mu}$ is a compactly supported measure, the moments uniquely determine it (via expansion of its characteristic function), and hence also $\mu$.
\end{proof}

\section{Acknowledgements}

I want to thank Assaf Naor for his encouragement and helpful discussions. I am grateful to my roommate for their meticulous proofreading. A significant part of the computational investigations was performed on Mathematica.

\bibliography{refs}

\bibliographystyle{alpha}

\end{document}